\documentclass[a4paper,12pt]{article}
\usepackage[T1]{fontenc}
\usepackage[latin1]{inputenc}
\usepackage[english]{babel}
\usepackage{babel}
\usepackage{amsmath,amsthm,amsfonts,amssymb}
\usepackage{makeidx}
\usepackage{enumerate}
\usepackage[normalem]{ulem}

\numberwithin{equation}{section}

\newtheorem{thm}[equation]{Theorem}
\newtheorem{cor}[equation]{Corollary}
\newtheorem{lem}[equation]{Lemma}
\newtheorem{prop}[equation]{Proposition}
\theoremstyle{definition}
\newtheorem{defn}[equation]{Definition}
\newtheorem{rem}[equation]{Remark}

\newcommand*{\m}[1]{\underline{#1}}
\newcommand*{\BAR}[1]{\overline{#1}}
\newcommand*{\HAT}[1]{\widehat{#1}}
\newcommand{\RE}{\operatorname{Re}}
\newcommand{\CO}{\operatorname{Co}}
\newcommand{\QU}{\operatorname{Qu}}
\newcommand{\IM}{\operatorname{Im}}
\newcommand{\GL}{\operatorname{GL}}
\newcommand{\SO}{\operatorname{SO}}
\newcommand{\U}{\operatorname{U}}
\newcommand{\Sp}{\operatorname{Sp}}
\newcommand{\OO}{\operatorname{O}}

\newcommand{\crxi}[1]{\partial_{x_{#1}}}
\newcommand{\crzi}[1]{\partial_{z_{#1}}}

\begin{document}

\title{Cauchy-Riemann Operators in Octonionic Analysis}
\author{Janne Kauhanen%
\footnote{Laboratory of Mathematics, Faculty of Natural Sciences, Tampere University of Technology, Finland. Electronic address: \texttt{janne.kauhanen@tut.fi}}
\ and Heikki Orelma%
\footnote{Electronic address: \texttt{heikki.orelma@tut.fi}}}

\maketitle

\begin{abstract}
In this paper we first recall the definition of an octonion algebra and its algebraic properties. We derive the so called $e_4$-calculus and using it we obtain the list of generalized Cauchy-Riemann systems in octonionic monogenic functions. We define some bilinear forms and derive the  corresponding symmetry groups. 
\end{abstract}

\section{Introduction}
The algebra of octonions is a well known non-associative division algebra. The second not so well known feature is, that we may define a function theory, in spirit of classical theory of complex holomorphic functions, and study its properties. This theory has its limitations, since the multiplication is neither commutative nor associative. The first part of this paper is a survey of known results, where we give a detailed definition for the octonions. We derive the so called ''$e_4$-calculus'' on it, to make our practical calculations easier. Then we recall the notion of the Cauchy-Riemann operator. A function in its kernel is called monogenic. To find explicit monogenic functions directly from the definitions is too complicated, because the algebraic properties give too many limitations. To give an explicit characterization of monogenic functions, we separate variables, or represent the target space as a direct sum of subalgebras. Using this trick we obtain a list of real, complex, and quaternionic partial differential equation systems, which are all generalizations of the complex Cauchy-Riemann system. These systems allow us to study explicit monogenic functions. We compute an example, assuming that the functions are biaxially symmetric.

Authors like to emphasize, that this work is the starting point for our future works on this fascinating field of mathematics. A reader should notice, that although the algebraic calculation rules look  really complicated, one may still derive a practical formulas to analyze the properties of the quantities of the theory. It seems that there are two possible ways to study the octonionic analysis in our sense. In the first one, one just takes results from classical complex or quaternionic analysis and tries to prove them. The second one is to concentrate to algebraic properties and features of the theory, and try to find something totally new, in the framework of the algebra. We believe that the latter gives us deeper intuition of the theory,  albeit the steps forward are not always so big.

\section{On Octonion Algebra}
In this section we recall the definition for the octonions and study its algebraic properties. We develop the so called $e_4$-calculus, which we will use during the rest of the paper to simplify practical computations. Also classical groups related to the octonions are studied.

\subsection{Definition of Octonions}
Let us denote the field of complex numbers by $\mathbb{C}$ and the skew field of quaternions by $\mathbb{H}$. We assume that the complex numbers are generated by the basis elements $\{1,i\}$ and the quaternions by $\{1,i,j,k\}$ with the well known defining relations
\[
i^2=j^2=k^2=ijk=-1.
\]
We expect that the reader is familiar with the complex numbers and the quaternions. We give \cite{CS,L,P} as a basic reference. The so called octonions or Cayley numbers were first defined defined in 1843 by John T. Graves. Nowadays the systematic way to define octonions is the so called Cayley-Dickson construction, which we will use also in this paper. See historical remarks on ways to define the octonions in \cite{B}.

% HUOLEHTIKOON LEHDEN TYYLITIEDOSTO KAPPALEENVAIHDON TYYLISTÄ, JÄTETÄÄN TYHJÄ RIVI KOODIIN
The Cayley-Dickson construction produces a sequence of algebras over the field of real numbers, each with twice the dimension of the previous one. The previous algebra of a Cayley-Dickson step is assumed to be an algebra with a conjugation. Starting from the algebra of real numbers $\mathbb{R}$ with the trivial conjugation $x\mapsto x$, the Cayley-Dickson construction produces the algebra of complex numbers $\mathbb{C}$ with the conjugation $x+iy\mapsto x-iy$. Then applying Cayley-Dickson construction to the complex numbers produces quaternions $\mathbb{H}$ with the conjugation. The quaternion conjugation is given as follows. An arbitrary $x\in\mathbb{H}$ is of the form
\[
x=x_0+\m{x}
\]
where $x_0\in\mathbb{R}$ is the real part and $\m{x}=x_1i+x_2j+x_3k$ is the vector part of the quaternion $x$. Vector parts are isomorphic to the three dimensional Euclidean vector space $\mathbb{R}^3$. Then the conjugation of $x$ obtained from the Cayley-Dickson construction is denoted by $\BAR{x}$ and defined by
\[
\BAR{x}=x_0-\m{x}.
\]
Now the Cayley-Dickson construction proceeds as follows. Consider pairs of quaternions, i.e., the space $\mathbb{H}\oplus\mathbb{H}$. We define the multiplication for the pairs as
\[
(a,b)(c,d)=(ac-\BAR{d}b, da+b\BAR{c})
\] 
where $a,b,c,d\in\mathbb{H}$. With this multiplication the pairs of quaternions $\mathbb{H}\oplus\mathbb{H}$ is an eight dimensional algebra generated by the elements
\begin{align*}
e_0:=(1,0),\ e_1:=(i,0),\ e_2:=(j,0),\ e_3:=(k,0),\\
e_4:=(0,1),\ e_5:=(0,i),\ e_6:=(0,j),\ e_7:=(0,k).
\end{align*}
Denoting $1:=e_0$ and using the definition of the product, we may write the following table.
\begin{center}
\begin{tabular}{c|cccccccc}
 & $1$   & $e_1$  & $e_2$  & $e_3$  & $e_4$  & $e_5$  & $e_6$  & $e_7$\\
\hline
$1$     & $1$   & $e_1$  & $e_2$  & $e_3$  & $e_4$  & $e_5$  & $e_6$  & $e_7$\\
$e_1$   & $e_1$ & $-1$ & $e_3$  & $-e_2$ & $e_5$  & $-e_4$ & $-e_7$ & $e_6$\\
$e_2$   & $e_2$ & $-e_3$ & $-1$ & $e_1$  & $e_6$  & $e_7$  & $-e_4$ & $-e_5$\\
$e_3$   & $e_3$ & $e_2$  & $-e_1$ & $-1$ & $e_7$  & $-e_6$ & $e_5$  & $-e_4$\\
$e_4$   & $e_4$ & $-e_5$ & $-e_6$ & $-e_7$ & $-1$   & $e_1$  & $e_2$  & $e_3$\\
$e_5$   & $e_5$ & $e_4$  & $-e_7$ & $e_6$  & $-e_1$ & $-1$ & $-e_3$ & $e_2$\\
$e_6$   & $e_6$ & $e_7$ & $e_4$  & $-e_5$ & $-e_2$ & $e_3$ & $-1$  & $-e_1$\\
$e_7$   & $e_7$ & $-e_6$ & $e_5$  & $e_4$ & $-e_3$ & $-e_2$ & $e_1$  & $-1$
\end{tabular}
\end{center}
We see that $e_0=1$ is the unit element of the algebra. Using the table, it is an easy task to see that the algebra is not associative nor commutative. Also we see that elements $\{1,e_1,e_2,e_3\}$ generates a quaternion algebra, i.e., $\mathbb{H}$ is a subalgebra. We will consider more subalgebras later.

% HUOLEHTIKOON LEHDEN TYYLITIEDOSTO KAPPALEENVAIHDON TYYLISTÄ, JÄTETÄÄN TYHJÄ RIVI KOODIIN
The preceding algebra  is called the \emph{algebra of octonions} and it is denoted by $\mathbb{O}$. An arbitrary $x\in\mathbb{O}$ may be represented in the form
\[
x=x_0+\m{x}
\]
where $x_0\in\mathbb{R}$ is the real part of the octonion $x$ and 
\[
\m{x}=x_1e_1+x_2e_2+x_3e_3+x_4e_4+x_5e_5+x_6e_6+x_7e_7,
\]
where $x_1,...,x_7\in\mathbb{R}$ is the vector part. Vector parts are isomorphic to the seven dimensional Euclidean vector space $\mathbb{R}^7$. The whole algebra of octonions is naturally identified as a vector space with $\mathbb{R}^8$.  The Cayley-Dickson construction produces also naturally in $\mathbb{O}$ a conjugation $(a,b)^*:=(\BAR{a},-b)$, where $\BAR{a}$ is the quaternion conjugation. Because there is no risk of confusion, we will denote the conjugation of $x\in\mathbb{O}$ by $\BAR{x}$. Using the definition, we have
\[
\BAR{x}=x_0-\m{x}.
\]
We refer \cite{B,CS,H} for more detailed description to the preceding construction.

\subsection{Algebraic Properties}
In this subsection we collect some algebraic properties and results of the octonions to better understand its algebraic structure.

\begin{prop}[$\mathbb{O}$ is an alternative division algebra, \cite{P}]\label{ALT}
If $x,y\in \mathbb{O}$ then
\[
x(xy)=x^2y,\ (xy)y=xy^2,\ (xy)x=x(yx),
\]
and each non-zero $x\in\mathbb{O}$ has an inverse.
\end{prop}
We see that the associativity holds in the case $(xy)x=x(yx)$. Unfortunately, this is almost the only non-trivial case when the associativity holds:

\begin{prop}[\cite{CS}]
If 
\[
x(ry)=(xr)y
\]
for all $x,y\in\mathbb{O}$, then $r$ is real.
\end{prop}
So we see, that use of parentheses is something what we need to keep in mind, when we compute using the octonions. The alternative properties given in Proposition \ref{ALT} implies the following identities.

\begin{prop}[Moufang Laws, \cite{CS,M}] For each $x,y,z\in\mathbb{O}$
\[
(xy)(zx)=(x(yz))x=x((yz)x).
\]
\end{prop}
The inverse element $x^{-1}$ of non-zero $x\in\mathbb{O}$ may be computed as follows. We define the norm by $|x|=\sqrt{x\BAR{x}}=\sqrt{\BAR{x}x}$. A straightforward computation shows that the norm is well defined and
\[
|x|^2=\sum_{j=0}^7 x_j^2.
\]
In addition,
\[
x^{-1}=\frac{\BAR{x}}{|x|^2}.
\]
An important property of the norm is the following.

\begin{prop}[$\mathbb{O}$ is a composition algebra, \cite{CS,H,P}]
The norm of $\mathbb{O}$ satisfies the \emph{composition law}
\[
|xy|=|x||y|
\]
for all $x,y\in\mathbb{O}$.
\end{prop}
We will say that octonions has a \emph{multiplicative norm}. The composition law has  algebraic implications for conjugation, since the conjugation may be written using the norm in the form $\BAR{x}=|x+1|^2-|x|^2-1-x$. 

\begin{prop}[\cite{CS}]If $x,y\in\mathbb{O}$, then
\[
\BAR{\BAR{x}}=x\ \text{ and }\ \BAR{xy}=\BAR{y}\,\BAR{x}.
\]
\end{prop}
These formulas are easy to prove by brute force computations. But the reader should notice, that actually they are consequences of the composition laws, not directly related only to octonions. In general we say that an algebra $A$ is a composition algebra, if it has a norm $N\colon A\to\mathbb{R}$ such that  $N(ab)=N(a)N(b)$ for all $a,b\in A$. We know that $\mathbb{R}$, $\mathbb{C}$, $\mathbb{H}$ and $\mathbb{O}$ are composition algebras. It is an interesting algebraic task to prove that actually this list is complete. 

\begin{thm}[Hurwitz, \cite{CS}]
$\mathbb{R}$, $\mathbb{C}$, $\mathbb{H}$ and $\mathbb{O}$ are the only composition algebras.
\end{thm}

\subsection{$e_4$--Calculus}
In this subsection we study how to compute with the octonions in practise. In principle all of the computations are possible to carry out using the multiplication table. In practise, this often leads to chaos of indices, so it is better to develop another kind of calculation. Our starting point is the observation that every octonion $x\in\mathbb{O}$ may be written in the form
\[
x=a+be_4
\]
where $a,b\in\mathbb{H}$. This form is called the \emph{quaternionic form} of an octonion. If
\[
x=x_0+x_1e_1+\cdots+x_7e_7,
\]
then
\[
a=x_0+x_1e_1+x_2e_2+x_3e_3\ \text{ and }\ b=x_4+x_5e_1+x_6e_2+x_7e_3.
\]
Using the multiplication table, it is easy to prove the following.
\begin{lem}\label{lem:eieje4}
Let $i,j\in\{1,2,3\}$. Then
\begin{enumerate}[(a)]
\item $e_i(e_je_4)=(e_je_i)e_4$,
\item $(e_ie_4)e_j=-(e_ie_j)e_4$,
\item $(e_ie_4)(e_je_4)=e_je_i$.
\end{enumerate}
\end{lem}
Using these, we have
\begin{lem}\label{lem:vecxvecy}
If $\m{a}=a_1e_1+a_2e_2+a_3e_3$ and $\m{b}=b_1e_1+b_2e_2+b_3e_3$, then
\begin{enumerate}[(a)]
\item $e_4\m{a}=-\m{a}e_4$\label{item:vexvecy3}
\item $e_4(\m{a}e_4)=\m{a}$\label{item:vexvecy4}
\item $(\m{a}e_4)e_4=-\m{a}$\label{item:vexvecy5}
\item $\m{a}(\m{b}e_4)=(\m{b}\,\m{a})e_4$\label{item:vexvecy6}
\item $(\m{a}e_4)\m{b}=-(\m{a}\,\m{b})e_4$\label{item:vexvecy7}
\item $(\m{a}e_4)(\m{b}e_4)=\m{b}\,\m{a}$\label{item:vexvecy8}
\end{enumerate}
\end{lem}
Using preceding formulae, it is easy to obtain similar formulas for quaternions.
\begin{lem}\label{lem:paravecprod}
Let $a,b\in\mathbb{H}$. Then
\begin{enumerate}
\item[(a)] $e_4a=\overline{a}e_4$
\item[(b)] $e_4(ae_4)=-\overline{a}$
\item[(c)] $(ae_4)e_4=-a$
\item[(d)] $a(be_4)=(ba)e_4$
\item[(e)] $(ae_4)b=(a\overline{b})e_4$
\item[(f)] $(ae_4)(be_4)=-\overline{b}a$
\end{enumerate}
\end{lem}
The preceding list is called the \emph{rules of $e_4$-calculus for the octonions}. When we compute using octonions, we drop our computations to quaternionic level and use the preceding formulas and associativity. The situation is similar to computing with complex numbers, where we usually compute by real numbers with the relation $i^2=-1$.

\begin{lem}\label{tulo}
Let $x=a_1+b_1e_4$ and $y=a_2+b_2e_4$ be octonions in the quaternionic form. Then their product in quaternionic form is
\begin{equation*}
xy=(a_1a_2-\overline{b}_2b_1)+(b_1\overline{a}_2+b_2a_1)e_4.
\end{equation*}
\end{lem}
\begin{proof}
Apply Lemma \ref{lem:paravecprod}:
\begin{align*}
xy&=(a_1+b_1e_4)(a_2+b_2e_4)\\
&=a_1a_2+(b_1e_4)a_2+a_1(b_2e_4)+(b_1e_4)(b_2e_4)\\
&=a_1a_2+ (b_1\overline{a}_2)e_4+(b_2a_1)e_4-\overline{b}_2b_1.\qedhere
\end{align*}
\end{proof}

\begin{lem}\label{lem:octconjnorm}
For an octonion $a+be_4$ in the quaternionic form we have
\begin{align*}
\overline{a+be_4}&=\overline{a}-be_4,\\
|a+be_4|^2&=|a|^2+|b|^2.
\end{align*}
\end{lem}

\subsection{Bilinear Forms and Their Invariance Groups}
In this section we consider some octonion valued bilinear forms and study their invariance groups. Some of the results are not new, but not well known. For the convenience of the reader we give the proofs here.

\subsubsection{On Involutions and Subalgebras}
Our aim here is to study how involutions and subalgebras of octonions are related to each other. We begin with the easiest case.

The real numbers $\mathbb{R}$ may be identified with the real parts of the octonions. Since the conjugation of an octonion $x=x_0+\m{x}$ is defined by $\BAR{x}=x_0-\m{x}$, the real part  $\RE(x)=x_0$ of the octonion $x$ may be computed as
\[
\RE(x)=\frac{1}{2}(x+\BAR{x}).
\]
Because $\RE^2=\RE$, it is a projection of the octonion algebra onto the real numbers. We make the following conclusion.

\begin{thm}
The real numbers $\mathbb{R}$ is a subalgebra of the octonions $\mathbb{O}$ generated by the identity element $\{1\}$. The mapping
\[
\RE\colon\mathbb{O}\to \mathbb{R}
\]
is a projection. 
\end{thm}
The complex numbers is a subalgebra of the octonions, and may be generated by any pair $\{1,e_j\}$ where $j=1,...,7.$ To find a canonical one, we define an involution
\[
x^*:=\BAR{a}+\BAR{b}e_4
\]
where $x=a+be_4\in\mathbb{O}$ is in quaternionic form. Then we define the complex part of an octonion by
\[
\CO(x):=\frac{1}{2}(x+x^*)=\RE(a)+\RE(b)e_4.
\]
Since $\CO^2=\CO$, it is a projection of the octonion algebra onto the complex numbers generated by $\{1,e_4\}$. We make the following conclusion.

\begin{thm}
The complex numbers $\mathbb{C}$ is a subalgebra of the octonions $\mathbb{O}$ generated by the elements $\{1,e_4\}$. The mapping
\[
\CO\colon\mathbb{O}\to \mathbb{C}
\]
is a projection. 
\end{thm}
The quaternions $\mathbb{H}$ is a subalgebra of the octonions $\mathbb{O}$, generated by any $\{1,e_i,e_j,e_ie_j\}$, where $i,j\in\{1,\ldots,7\}$, $i\neq j$. If $x=a+be_4\in\mathbb{O}$ we define an involution
\[
\HAT{x}:=a-be_4.
\]
Using hat, we define the quaternion part of an octonion $x$ as
\[
\QU(x):=\frac{1}{2}(x+\HAT{x}),
\]
that is, $\QU(a+be_4)=a$. Since $\QU^2=\QU$ we make the following conclusion.

\begin{thm}
The quaternions $\mathbb{H}$ is a subalgebra of the octonions $\mathbb{O}$ generated by the  elements $\{1,e_1,e_2,e_3\}$. The mapping
\[
\QU\colon\mathbb{O}\to \mathbb{H}
\]
is a projection. 
\end{thm}
Next we study algebraic properties of the preceding involutions.

\begin{prop} If $x,y\in\mathbb{O}$, then
\begin{enumerate}[(a)]
\item $x^{**}=x$,\label{item:algpropa}
% \item $(xy)^{*}=??$ (TATA SAAPI MIETTIA)\label{item:algpropb}
% \textcolor{red}%
% {If $a,b\in\mathbb{H}$, then
% \begin{align*}
% (ab)^*&=b^*a^*\\
% ((ae_4)(be_4))^*&=-b(ae_4)^*=-(ae_4)^*b^*\\
% (a(be_4))^*&=b^*(ae_4)^*=(ae_4)^*b\\
% ((ae_4)b)^*&=a^*(be_4)=(be_4)a
% \end{align*}
% Eli ei mitaan kovin kayttokelpoista. Tata tonkiessa huomasin seuraavan, seuraa version 24.8.2016 kaavasta (2.6). $x,y\in\mathbb{O}$:
% \[\BAR{x}\,\BAR{y}-\BAR{y}\,\BAR{x}=xy-yx=2\underline{x}\times\underline{y}.\]
% }
\item $\HAT{\HAT{x}}=x$,\label{item:algpropc}
\item $\HAT{xy}=\HAT{x}\,\HAT{y}$.\label{item:algpropd}
\end{enumerate}
\end{prop}
\begin{proof}
Parts \eqref{item:algpropa} and \eqref{item:algpropc} are obvious. We prove \eqref{item:algpropd}. Let $x=a_1+b_1e_4$ and $y=a_2+b_2e_4$. Then using the rules of $e_4$-calculus and Lemma \ref{tulo}, we have
\begin{align*}
\HAT{x}\,\HAT{y}&=(a_1-b_1e_4)(a_2-b_2e_4)\\
&=a_1a_2-(b_1e_4)a_2-a_1(b_2e_4)+(b_1e_4)(b_2e_4)\\
&=a_1a_2-\BAR{b}_2b_1-(b_1\BAR{a}_2+b_2a_1)e_4=\HAT{xy}.\qedhere
\end{align*}
\end{proof}

\subsubsection{Linear Mappings and their Invariance Groups}
A real linear mapping $T\colon\mathbb{O}\to\mathbb{O}$ from the octonions into itself is  acting on
\[
x=\sum_{j=0}^7 x_je_j
\]
as
\[
Tx=\sum_{i,j=0}^7 T_{ij}x_je_i
\]
where $T_{ij}\in\mathbb{R}$. We define the matrix representation of $T$ by $[T]:=[T_{ij}]\in\mathbb{R}^{8\times 8}$. It is easy to see that
$[TS]=[T][S]$ and $[T^{-1}]=[T]^{-1}$. Using matrix representation we may define the determinant of $T$ as $\det(T):=\det([T])$. The set of invertible linear mappings on $\mathbb{O}$ is denoted by $\GL(\mathbb{O})$.

We will consider real bilinear functions
\[
B\colon\mathbb{O}\times\mathbb{O}\to\mathbb{O}.
\]
For each function, we may associate a symmetry group
\[
\mathcal{G}(B):=\{ T\in\GL(\mathbb{O})\colon B(Tx,Ty)=B(x,y)\ \forall x,y\in\mathbb{O}\}.
\]
Let us next study bilinear functions generated by preceding projection mappings and the product $x\BAR{y}$.

\subsubsection{$\SO(8)$}
We define a bilinear form $B_{\mathbb{R}}\colon\mathbb{O}\times\mathbb{O}\to\mathbb{R}$ by
\[
B_{\mathbb{R}}(x,y)=\RE(x\BAR{y}).
\]
One may compute
\[
B_{\mathbb{R}}(x,y)=\sum_{j=0}^7 x_jy_j,
\]
and it is well known (see e.g. \cite{H,P}), that this form is invariant under orthogonal transformations, i.e., 
\[
\mathcal{G}(B_{\mathbb{R}})=\SO(8).
\]

\subsubsection{$\U(4)$}
This part is a modification of a similar result in \cite{G}, but a different definition of the octonions was used and none of the proofs were given. We define a bilinear form $B_{\mathbb{C}}\colon\mathbb{O}\times\mathbb{O}\to\mathbb{C}$ by
\[
B_{\mathbb{C}}(x,y)=\CO(x\BAR{y}).
\]
If $z=u+ve_4\in\mathbb{C}$, we denote $\RE(z)=u$ and $\IM(z)=v$. Then $B_{\mathbb{C}}=\RE (B_{\mathbb{C}})+\IM (B_{\mathbb{C}}) e_4$ and
\[
\mathcal{G}(B_{\mathbb{C}})=\mathcal{G}(\RE (B_{\mathbb{C}}))\cap \mathcal{G}(\IM (B_{\mathbb{C}})).
\]

\begin{lem} If $B_{\mathbb{C}}$ is the preceding bilinear function, then
\begin{enumerate}[(a)]
\item $\mathcal{G}(\RE (B_{\mathbb{C}}))=\SO(8)$,
\item $\mathcal{G}(\IM (B_{\mathbb{C}}))=\Sp(8)$.
\end{enumerate}
\end{lem}
\begin{proof}
(a) This follows from the fact $\RE(B_\mathbb{C})=B_\mathbb{R}$.\\
% Assume $x=a_1+b_1e_4$ and $y=a_2+b_2e_4$.
% Using Lemma \ref{tulo} and \ref{lem:octconjnorm}, we obtain
% \begin{equation*}
% x\BAR{y}=(a_1\BAR{a}_2+\overline{b}_2b_1)+(b_1a_2-b_2a_1)e_4
% \end{equation*}
% and then we know from properties of Quaternions that
% \begin{align*}
% \RE (B_{\mathbb{C}})&=\RE (a_1\BAR{a}_2+\overline{b}_2b_1)=x_0y_0+x_1y_1+\cdots+x_7y_7.
% \end{align*}
(b) Writing
\begin{align*}
\RE(b_1a_2)=x_4y_0-x_5y_1-x_6y_2-x_7y_3,\\
\RE(b_2a_1)=x_0y_4-x_1y_5-x_2y_6-x_3y_7,
\end{align*}
we have
\begin{align*}
&\IM(B_{\mathbb{C}}(x,y))=\RE(b_1a_2-b_2a_1)\\
&=x_4y_0-x_5y_1-x_6y_2-x_7y_3-x_0y_4+x_1y_5+x_2y_6+x_3y_7\\
&=[x]^T\begin{bmatrix}
&&&&-1\\
&&&&&1\\
&&&&&&1\\
&&&&&&&1\\
1\\
&-1\\
&&-1\\
&&&-1
\end{bmatrix}[y],
\end{align*}
where $[x]^T=\begin{bmatrix}x_0 & x_1 & \cdots & x_7\end{bmatrix}$ and $[y]$ similarly. Changing the coordinates $x_0$ and $x_4$, and $y_0$ and $y_4$, we see that 
\[
\IM(B_{\mathbb{C}}(x,y))=[x]^T\begin{bmatrix}
& I_4\\
-I_4
\end{bmatrix}[y].
\]
It is well known that symplectic transformations leaves this form invariant.
\end{proof}

Using classical 2-out-of-3 property\footnote{This means, that in general the unitary group is the intersection $\U(n)=\OO(2n)\cap\Sp(2n,\mathbb{R})\cap\GL(n,\mathbb{C})$.} we obtain, that
\[
\mathcal{G}(B_{\mathbb{R}})=\U(4).
\]

\subsubsection{$S^4_R\times S^4_L$}
Let us now consider bilinear function
$B_{\mathbb{H}}\colon\mathbb{O}\times\mathbb{O}\to\mathbb{H}$ by
\[
B_{\mathbb{H}}(x,y)=\QU(x\BAR{y}).
\]
If $x=a_1+b_1e_4$ and $y=a_2+b_2e_4$, using Lemmas \ref{tulo} and \ref{lem:octconjnorm} we obtain
\begin{equation*}
B_{\mathbb{H}}(x,y)=a_1\BAR{a}_2+\overline{b}_2b_1.
\end{equation*}
Let $S^3=\{x\in\mathbb{H}\colon x\BAR{x}=\BAR{x}x=1\}$ be the unit $3$-sphere. We obtain the observation (see \cite{G}), that
\[
a_1\BAR{a}_2+\overline{b}_2b_1=a_1q\BAR{a_2q}+\overline{pb_2}pb_1
\]
for all $q,p\in S^3$. Let us recall the following result.

\begin{prop}\cite[Proposition 8.26]{P}
Let $q\in S^3$. Define the mappings $L_q$ and $R_q$ by
\[
L_q(a)=qa\ \text{ and }\ R_q(a)=aq,
\]
where $a\in\mathbb{H}$. Then $L_q,R_q\in\SO(4)$.
\end{prop}
Define the groups $S^3_L:=\{ L_q\colon q\in S^3\}$ and $S^3_R:=\{ R_q\colon q\in S^3\}$. We may define the action on octonions as
\[
\rho_{q,p}(a+be_4):=R_qa+(L_pb)e_4.
\]
Then we see, that $B_{\mathbb{H}}(\rho_{q,p}x,\rho_{q,p}y)=B_{\mathbb{H}}(x,y)$, and 
\[
S^3_R\times S^3_L\subset\mathcal{G}(B_{\mathbb{H}}).
\]

\section{Cauchy-Riemann Operators}
In this section we begin to study the basic analytical properties of the octonion valued functions. First we recall some basic properties and after that we express some equivalent systems related to the decomposition of octonions. Using these equivalent systems we may avoid non-associativity.

\subsection{Definitions and Basic Properties}
In the octonionic analysis we consider functions defined on a set $\Omega\subset\mathbb{R}^8\cong\mathbb{O}$ and taking values in $\mathbb{O}$. Similarly than in the case of quaternionic analysis, we may consider octonionic analyticity, and see that the generalization of Cauchy-Riemann equations is the only way to get a nice function class (see \cite{LKP}). We begin by connecting to an octonion
\[
x=x_0+x_1e_1+\cdots+x_7e_7
\]
the derivative operator
\[
\partial_x=\partial_{x_0}+e_1\partial_{x_1}+\cdots+e_7\partial_{x_7}.
\]
The preceding derivative operator is called the \emph{Cauchy-Riemann operator}. The vector part of it
\[
\partial_{\m{x}}= e_1\partial_{x_1}+\cdots+e_7\partial_{x_7}
\]
is called the \emph{Dirac operator}. Now it is easy to represent the Cauchy-Riemann operator and its conjugate as
\[
\partial_x=\partial_{x_0}+\partial_{\m{x}}\ \text{ and }\ \partial_{\BAR{x}}=\partial_{x_0}-\partial_{\m{x}}.
\] 
The function $f\colon\mathbb{O}\to\mathbb{O}$ is of the form
\[
f=\sum_{j=0}^7 e_jf_j
\]
where $f\colon\mathbb{O}\to\mathbb{R}$. If the components of $f$ have partial derivatives, then $\partial_x$ operates from the left as
\[
\partial_xf
=\sum_{i=0}^7e_i\partial_{x_i}f
=\sum_{i,j=0}^7e_ie_j\partial_{x_i}f_j
\]
and from the right as
\[
f\partial_x
=\sum_{i=0}^7fe_i\partial_{x_i}
=\sum_{i,j=0}^7e_je_i\partial_{x_i}f_j.
\]

\begin{defn}\label{def:monogenity}
Let $\Omega\subset\mathbb{O}$ be open and $f\colon\Omega\to\mathbb{O}$ componentwise differentiable function. If
\[
\partial_xf=0\quad (\text{resp. }f\partial_x=0)
\]
in $\Omega$, then $f$ is called \emph{left (resp. right) monogenic in $\Omega$}.
%We say that $f$ is \emph{monogenic}, if it is left or right monogenic.
\end{defn}
We define the \emph{Laplace operator} as
\[
\Delta_x=\partial_{x_0}^2+\partial_{x_1}^2+\cdots+\partial_{x_7}^2.
\]
Because $\BAR{x}x=x\BAR{x}$, it follows that in $C^2(\mathbb{O},\mathbb{O})$
\begin{equation}
\partial_{\BAR{x}}\partial_x=\partial_x\partial_{\BAR{x}}=\Delta_x.
\end{equation}
From the alternativity (Proposition \ref{ALT}) it follows $(\BAR{x}x)y=\BAR{x}(xy)$, and therefore for $f\in C^2(\mathbb{O},\mathbb{O})$
\begin{equation}
(\partial_{\BAR{x}}\partial_x)f
=\partial_{\BAR{x}}(\partial_xf)
=\partial_x(\partial_{\BAR{x}}f).
\end{equation}
Similarly
\begin{equation}
f(\partial_{\BAR{x}}\partial_x)
=(f\partial_{\BAR{x}})\partial_x
=(f\partial_x)\partial_{\BAR{x}}.
\end{equation}

These properties give us, like in the quaternionic analysis case:

\begin{prop}\label{prop:mongisharm}
Let a function $f\in C^2(\mathbb{O},\mathbb{O})$ be left or right monogenic. Then $f$ is harmonic.
\end{prop}

%are harmonic in $\mathbb{R}^8$, i.e., $\Delta_x f_j=0$ for each $j=0,1,...,7$. 

Some basic function theoretical result have already been studied in octonionic analysis, e.g., the following classical integral formulas holds.

\begin{thm}[\cite{LP}]
Let $M$ be an 8-dimensional, compact, oriented smooth manifold with boundary $\partial M$ contained in some open connected subset $\Omega\subset\mathbb{R}^8$, and the function $f\colon\Omega\to \mathbb{O}$ left monogenic. Then for each $x\in M$ we have
\[
f(x)=\frac{1}{\omega_8}\int_{\partial M}\frac{\BAR{x-y}}{|x-y|^8} (n(y)f(y))\,dS(y),
\]
where $\omega_8$ is the volume of the sphere $S^7$, $n$ outward pointing unit normal on $\partial M$ and $dS$ the scalar surface element on the boundary.
\end{thm}
Using this theorem, similarly than in the quaternionic analysis case, we may prove many function theoretic results, for example the mean value theorem, maximum modulus theorem and Weierstrass type approximation theorems, see \cite{LP}.

\subsection{Equivalent Systems for Monogenic Functions}
In this paper our aim is to understand little bit better what are the monogenic functions in the octonionic analysis. In this section we consider two equivalent formulation for the equation $\partial_xf=0$ and make some observations. These elementary observations motivate us to consider more interesting case in the next section.

\subsubsection{A Real Decomposition}
We start from the most trivial case. We observe that the octonion algebra may be represented as a direct sum of $1$-dimensional real subspaces
\[
\mathbb{O}=\bigoplus_{j=0}^7 e_j\mathbb{R}.
\]
Now we can separate the variables and split also the target space and the Cauchy-Riemann operator due to this decomposition and write the variables, the functions, and the Cauchy-Riemann operator in the form
\begin{align*}
x&=\sum_{j=0}^7x_je_j,\\
f&=\sum_{j=0}^7f_je_j,\\
\partial_x&=\sum_{j=0}^7 e_j\partial_{x_j}.
\end{align*}
A tedious but straightforward computation yields that $f$ is left monogenic if and only if its component functions $f_0,f_1,...,f_7$ satisfies the $8\times 8$ real partial differential equation system
\newcommand{\parder}[2]{\partial_{x_#1}f_#2}
\[
\begin{cases}
\partial_{x_0}f_0-\partial_{x_1}f_1-\ldots-\partial_{x_7}f_7=0\\
\parder{0}{1}+\parder{1}{0}
+\parder{2}{3}-\parder{3}{2}
+\parder{4}{5}-\parder{5}{4}
-\parder{6}{7}+\parder{7}{6}=0\\
\parder{0}{2}+\parder{2}{0}
-\parder{1}{3}+\parder{3}{1}
+\parder{4}{6}-\parder{6}{4}
+\parder{5}{7}+\parder{7}{5}=0\\
\parder{0}{3}+\parder{3}{0}
+\parder{1}{2}-\parder{2}{1}
+\parder{4}{7}-\parder{7}{4}
-\parder{5}{6}+\parder{6}{5}=0\\
\parder{0}{4}+\parder{4}{0}
-\parder{1}{5}+\parder{5}{1}
-\parder{2}{6}+\parder{6}{2}
-\parder{3}{7}+\parder{7}{3}=0\\
\parder{0}{5}+\parder{5}{0}
+\parder{1}{4}-\parder{4}{1}
-\parder{2}{7}+\parder{7}{2}
+\parder{3}{6}-\parder{6}{3}=0\\
\parder{0}{6}+\parder{6}{0}
+\parder{1}{7}-\parder{7}{1}
+\parder{2}{4}-\parder{4}{2}
-\parder{3}{5}+\parder{5}{3}=0\\
\parder{0}{7}+\parder{7}{0}
-\parder{1}{6}+\parder{6}{1}
+\parder{2}{5}-\parder{5}{2}
+\parder{3}{4}-\parder{4}{3}=0
\end{cases}
\]
\begin{rem}
A reader should notice that this system is different from the Riesz system of Stein and Weiss, 
\[%\label{eq:rieszcl}
\begin{cases}
\partial_{x_0}f_0-\partial_{x_1}f_1-\ldots-\partial_{x_7}f_7=0\\
\partial_{x_0}f_i+\partial_{x_i}f_0=0\qquad(i=1,\ldots,7)\\
\partial_{x_i}f_j-\partial_{x_j}f_i=0\qquad(i,j=1,\ldots,7,\ i\ne j)
\end{cases}
\]
\end{rem}

\subsubsection{A Complex Decomposition}
The preceding case motivates us to proceed further using similar techniques. Now we observe that the octonions may be express as a direct sum of complex numbers
\[
\mathbb{O}=\mathbb{C}\oplus \mathbb{C}e_2\oplus(\mathbb{C}\oplus \mathbb{C}e_2)e_4,
\]
where a basis of $\mathbb{C}$ is $\{1,e_1\}$. We may write an octonion with respect to this decomposition as
\begin{equation*}
x=z_1+z_2e_2+(z_3+z_4e_2)e_4.
\end{equation*}
where we denote
\begin{equation*}
\begin{aligned}
z_1&=x_0+x_1e_1,\\
z_2&=x_2+x_3e_1,\\
z_3&=x_4+x_5e_1,\\
z_4&=x_6+x_7e_1.
\end{aligned}
\end{equation*}
Similarly we express a function $f$ as a sum of complex valued functions $f_j=f_j(z_1,z_2,z_3,z_4)$ in the form
\begin{equation}
f=f_1+f_2e_2+(f_3+f_4e_2)e_4.
\end{equation}
If we define complex Cauchy-Riemann operators as 
\begin{align*}
\crzi{1}&=\crxi{0}+e_1\crxi{1},\\
\crzi{2}&=\crxi{2}+e_1\crxi{3},\\
\crzi{3}&=\crxi{4}+e_1\crxi{5},\\
\crzi{4}&=\crxi{6}+e_1\crxi{7},
\end{align*}
we may split the Cauchy-Riemann operator as
\[
\partial_x=\crzi{1}+\crzi{2} e_2+(\crzi{3}+\crzi{4} e_2)e_4.
\]
Again, after tedious computations, one have that $\partial_xf=0$ is equivalent to the complex $4\times 4$ equation system
\[
\begin{cases}
\crzi{1}f_1-\crzi{2}\BAR{f}_2-\crzi{3}\BAR{f}_3-\partial_{\BAR{z}_4}f_4=0\\
\crzi{1}f_2+\crzi{2}\BAR{f}_1+\partial_{\BAR{z}_3}f_4-\crzi{4}\BAR{f}_3=0\\
\crzi{1}f_3-\partial_{\BAR{z}_2}f_4+\crzi{3}\BAR{f}_1+\crzi{4}\BAR{f}_2=0\\
\partial_{\BAR{z}_1}f_4+\crzi{2}f_3-\crzi{3}f_2+\crzi{4}f_1=0
\end{cases}
\]

\subsection{Quaternionic Cauchy-Riemann Equations}
The preceding observations motivates us to consider the following case. We express the octonion algebra as a direct sum of quaternions
\[
\mathbb{O}=\mathbb{H}\oplus \mathbb{H}e_4.
\]
This decomposition correspond quaternionic forms of octonions and we know that every function takes the form $f=g+he_4$. If we split also a variable $x=u+ve_4$, where
\[
u=x_0+x_1e_1+x_2e_2+x_3e_3\ \text{ and }\  v=x_4+x_5e_1+x_6e_2+x_7e_3,
\]
we observe that $f,g\colon\mathbb{H}\times\mathbb{H}\to\mathbb{H}$ are  functions of two quaternionic variables. 
Similarly we split
\[
\partial_x=\partial_u+\partial_ve_4,
\]
where $\partial_u$ and $\partial_v$ are quaternionic Cauchy-Riemann operators. The rules of $e_4$-calculus in Lemma \ref{lem:paravecprod} give us immediately:

\begin{lem}\label{lem:diffrules}
Let $f\colon\Omega\subset\mathbb{H}\to\mathbb{H}$ be a differentiable function and $\partial_u=\partial_{u_0}+e_1\partial_{u_1}+e_2\partial_{u_2}+e_3\partial_{u_3}$ the quaternionic Cauchy-Riemann operator. Then we have
\begin{enumerate}[(a)]
\item $\partial_u(fe_4)=(f\partial_u)e_4$,
\item $(\partial_u e_4)f=(\partial_u\overline{f})e_4$,
\item $(\partial_u e_4)(fe_4)=-\overline{f}\partial_u$.
\item $(fe_4)\partial_u=(f\partial_{\BAR{u}})e_4$,
\item $f(\partial_u e_4)=(\partial_uf)e_4$,
\item $(fe_4)(\partial_u e_4)=-\partial_{\BAR{u}}f$.
\end{enumerate}
\end{lem}
Using these rules, we obtain the following equivalent systems.

\begin{prop}[Quaternionic Cauchy-Riemann systems]
Assume that $f=g+he_4$ is as above. Then
\begin{itemize}
\item[(a)] $\partial_xf=0$ if and only if
\begin{equation}\label{eq:qcr1}
\begin{cases}
\partial_ug=\BAR{h}\partial_v,\\
h\partial_u=-\partial_v\BAR{g}.
\end{cases}
\end{equation}
\item[(b)] $f\partial_x=0$ if and only if
\begin{equation}\label{eq:qcr2}
\begin{cases}
g\partial_u=\partial_{\BAR{v}}h,\\
h\partial_{\BAR{u}}=-\partial_v g.
\end{cases}
\end{equation}
\end{itemize}
\end{prop}
\begin{proof}
Let $f=g+he_4$ and $\partial_x=\partial_u+\partial_ve_4$. Then we have
\begin{align*}
\partial_x f&=(\partial_u+\partial_ve_4)(g+he_4)\\
&=\partial_ug+\partial_u(he_4)+(\partial_ve_4)g+(\partial_ve_4)(he_4)\\
&=\partial_ug+(h\partial_u)e_4+(\partial_v\BAR{g})e_4-\BAR{h}\partial_v,
\end{align*}
which gives us (a). Computations for (b) are similar.
\end{proof}
As a special case:
\begin{cor} 
\begin{itemize}
\item[(a)] If $\m{g}=\m{h}=0$, then $\partial_xf=0$ if and only if
\[
\begin{cases}
\partial_ug_0=\partial_vh_0,\\
\partial_uh_0=-\partial_vg_0.
\end{cases}
\]
\item[(b)] If $g_0=h_0=0$, then $\partial_xf=0$ if and only if
\[
\begin{cases}
\partial_u\m{g}=-\m{h}\partial_v,\\
\m{h}\partial_u=\partial_v\m{g}.
\end{cases}
\]
\end{itemize}
\end{cor}
Proposition \ref{prop:mongisharm} implies:
\begin{prop}\label{prop:hgharmonic}
If $g$ and $h\colon\mathbb{H}\times\mathbb{H}\to\mathbb{H}$ satisfy Cauchy-Riemann system \eqref{eq:qcr1} or \eqref{eq:qcr2}, then $g$ and $h$ are harmonic.
\end{prop}
In this point of view, octonionic analysis is actually two variable quaternionic analysis and there is a natural biaxial behaviour. Since $\mathbb{O}$ is an alternating algebra, i.e., $x(yx)=(xy)x$, we may also define \emph{inframonogenic functions} as in the classical case (see \cite{MPS}) as functions which satisfy the system $\partial_x f\partial_x=0$. For inframonogenic functions we obtain the following equivalent decomposition.

\begin{prop}\label{prop:inframon}
A function $f=g+he_4\in C^2(\mathbb{O},\mathbb{O})$ is inframonogenic if and only if
\begin{align*}
\Delta_v\BAR{g}-\partial_ug\partial_u+\BAR{h}\partial_v\partial_u+\partial_{\BAR{v}}h\partial_u=0,\\
\Delta_uh+\partial_v\BAR{g}\partial_{\BAR{u}}+\partial_v\partial_ug-\partial_v\BAR{h}\partial_v=0.
\end{align*}
\end{prop}
\begin{proof}
As above,
\begin{align*}
\partial_x f&=q_1+q_2e_4,
\end{align*}
where $q_1=\partial_ug-\BAR{h}\partial_v$ and $q_2=\partial_v\BAR{g}+h\partial_u$.
Using the differentiation rules of Lemma \ref{lem:diffrules} we compute
\begin{align*}
\partial_x f\partial_x&=(q_1+q_2e_4)(\partial_u+\partial_ve_4)\\
&=q_1\partial_u+(q_2e_4)\partial_u+q_1(\partial_ve_4)+(q_2e_4)(\partial_ve_4)\\
&=q_1\partial_u+(q_2\partial_{\BAR{u}})e_4+(\partial_v q_1)e_4-\partial_{\BAR{v}}q_2\\
&=\partial_ug\partial_u-\BAR{h}\partial_v\partial_u-\Delta_v\BAR{g}-\partial_{\BAR{v}}h\partial_u\\
&\quad+(\partial_v\BAR{g}\partial_{\BAR{u}}+\Delta_uh+\partial_v\partial_ug-\partial_v\BAR{h}\partial_v)e_4.\qedhere
\end{align*}
\end{proof}

\subsection{Real Biaxially Radial Solutions -- a Connection to Holomorphic Functions}
We close the paper studying the following example. Let us consider real valued  functions $g$ and $h\colon\mathbb{O}\to\mathbb{R}$ which are axially symmetric, i.e., invariant under the action of the spingroup: for all $q\in S^3$
\begin{equation}
g(u_0,\m{u},v_0,\m{v})=qg(u_0,\BAR{q}\m{u}q,v_0,\BAR{q}\m{v}q)\BAR{q}.
\end{equation}
%radial with respect to the axis $\m{u}$ and $\m{v}$
Then they depend only on $u_0$, $v_0$, $a=|\m{u}|^2$, $b=|\m{v}|^2$ and $c=\langle \m{u},\m{v}\rangle$, see \cite[Section 4]{Sommen}. The function $f=g+he_4$ is monogenic if $g$ and $h$ satisfy the system
\[
\begin{cases}
\partial_ug=\partial_vh,\\
\partial_uh=-\partial_vg.
\end{cases}
\]
Substituting
\begin{align*}
\partial_{\m{u}}g=2\m{u}\partial_ag+\m{v}\partial_cg,\\
\partial_{\m{v}}g=2\m{v}\partial_bg+\m{u}\partial_cg,
\end{align*}
and similarly for $h$, we obtain the system
\begin{align*}
\begin{cases}
2\partial_ag-\partial_ch=0,\\
\partial_cg-2\partial_bh=0,\\
2\partial_ah+\partial_cg=0,\\
\partial_ch+2\partial_bg=0,\\
\partial_{u_0}g-\partial_{v_0}h=0,\\
\partial_{u_0}h+\partial_{v_0}g=0.
\end{cases}
\end{align*}
The first four equations gives us 
\begin{align*}
\partial_a g+\partial_bg=0,\\
\partial_a h+\partial_bh=0.
\end{align*}
Let us look for a solution of the form
\[
g=G(u_0,v_0,a-b,c)\ \text{ and }\ h= H(u_0,v_0,a-b,c).
\]
Substituting these, the first four equations take the form
\begin{align*}
2\partial_dG-\partial_cH=0,\\
2\partial_dH+\partial_cG=0,
\end{align*}
where we denote $d=a-b$. Let us look for a solution in the form $H=e^dp(c,u_0,v_0)$ and $G=e^d q(c,u_0,v_0)$. We get
\begin{align*}
2q-\partial_cp=0,\\
2p+\partial_cq=0.
\end{align*}
The system has a solution
\begin{align*}
p(c,u_0,v_0)&=-\alpha(u_0,v_0)\cos(2c)+\beta(u_0,v_0)\sin(2c),\\
q(c,u_0,v_0)&=\alpha(u_0,v_0)\sin(2c)+\beta(u_0,v_0)\cos(2c),
\end{align*}
i.e.,
\begin{align*}
g&=\alpha(u_0,v_0)e^d\sin(2c)+\beta(u_0,v_0)e^d\cos(2c),\\
h&=-\alpha(u_0,v_0)e^d\cos(2c)+\beta(u_0,v_0)e^d\sin(2c).
\end{align*}
Substituting these to 
\begin{align*}
\partial_{u_0}g-\partial_{v_0}h&=0,\\
\partial_{u_0}h+\partial_{v_0}g&=0,
\end{align*}
we obtain
\begin{align*}
\partial_{u_0}\alpha&=\partial_{v_0}\beta,\\
\partial_{v_0}\alpha&=-\partial_{u_0}\beta.
\end{align*}
This is the Cauchy-Riemann system. We see that if $\alpha+i\beta$ is a holomorphic function, it gives us a monogenic function
\begin{align*}
f(x)=&e^{|\m{u}|^2-|\m{v}|^2}\Big(\alpha(u_0,v_0)\sin(2\langle \m{u},\m{v}\rangle)+\beta(u_0,v_0)\cos(2\langle \m{u},\m{v}\rangle)\\
&+\big(-\alpha(u_0,v_0)\cos(2\langle \m{u},\m{v}\rangle)+\beta(u_0,v_0)\sin(2\langle \m{u},\m{v}\rangle\big)e_4\Big).
\end{align*}

\end{document}